\documentclass[11pt]{amsart}
\usepackage[mathscr]{eucal}
\usepackage{palatino, mathpazo, amsfonts, mathrsfs, amscd}
\usepackage[all]{xy}
\usepackage{url}
\usepackage{amssymb, amsmath, amsthm}
\usepackage{stackrel}
\usepackage{bm}

\usepackage{tikz-cd}
\usepackage{marginnote}

\newtheorem{theorem}{Theorem}[section]
\newtheorem{lemma}[theorem]{Lemma}

\newtheorem{proposition}[theorem]{Proposition}
\newtheorem{corollary}[theorem]{Corollary}

\theoremstyle{remark}
\newtheorem{remark}[theorem]{Remark}

\newtheorem{definition}[theorem]{Definition}
\newtheorem{notation}[theorem]{Notation}

\numberwithin{equation}{subsection}
\usepackage{todonotes}

\makeatletter
\def\imod#1{\allowbreak\mkern10mu({\operator@font mod}\,\,#1)}
\makeatother

\newcommand{\spec}{\operatorname{Spec}}

\newcommand{\CC}{\mathbb{C}}
\newcommand{\ZZ}{\mathbb{Z}}
\newcommand{\LL}{\mathbb{L}}
\newcommand{\KK}{\mathbb{K}}

\newcommand{\QQ}{\mathbb{Q}}

\newcommand{\RR}{\mathbb{R}}

\newcommand{\s}[1]{\mathscr{#1}}  
\newcommand{\cc}[1]{\mathcal{#1}}  

\newcommand{\sMbar}{\overline{\mathscr{M}}}

\newcommand{\sA}{{\mathscr{A}}}
\newcommand{\cL}{{\mathcal{L}}}

\newcommand{\bq}{\mathbf{q}}

\newcommand{\bt}{\bm{t}}

\newcommand{\btau}{\bm{\tau}}
\newcommand{\ii}{\mathbb{1}}
\newcommand{\iI}{\mathbb{I}}

\DeclareMathOperator{\Hom}{Hom}

\newcommand{\br}[1]{\left\langle#1\right\rangle}  
\newcommand{\set}[1]{\left\{#1\right\}}  

\title[Gromov--Witten Theory of Toric Birational Transformations]{Gromov--Witten Theory of Toric Birational Transformations}
\author{Pedro Acosta and Mark Shoemaker}

\begin{document}
%

\begin{abstract}
We investigate the effect of a general toric wall crossing on genus zero Gromov--Witten theory.    
Given two complete toric orbifolds $X_+$ and $X_-$ related by wall crossing under variation of GIT, we prove that their respective $I$-functions are related by linear transformation and asymptotic expansion. 
We use this comparison to deduce a similar result for birational complete intersections in $X_+$ and $X_-$. This extends the work of the previous authors in \cite{AS} to the case of complete intersections in toric varieties, and
generalizes some of the results of Coates--Iritani--Jiang \cite{CIJ} on the crepant transformation conjecture to the setting of non-zero discrepancy.  
\end{abstract}

\maketitle
{\small
\tableofcontents}

\section{Introduction}
A longstanding goal in Gromov--Witten theory has been to understand how Gromov--Witten invariants change under 
birational transformation.  Results in this direction have generally taken one of two forms. 
Namely, they focus on how a birational transformation affects either (a) the individual Gromov--Witten invariants of a space \cite{Ga, Hu, HH, GW, ILLW}, or (b)  certain formal structures defined via Gromov--Witten theory, e.g. quantum cohomology, generating functions, Lagrangian cones, etc... \cite{LR, BCS, BCR, CIT, CI, CIJ, Zhou}.  

Results of the first type are often more general in terms of the cases they cover.  The most systematic treatment along these lines of which we are aware is \cite{GW}, which deals with wall crossing under general variation of GIT.  Theorems of the second type are important however, as they show how certain Gromov--Witten-theoretic structures may be preserved under birational transformation even as individual invariants are not.  As an example of this phenomenon, in some instances the generating functions of Gromov--Witten invariants of birational spaces are related by analytic continuation \cite{BG}.  
These types or results are predicted by the \emph{crepant transformation conjecture} \cite{CIT, CR}, but as such, deal
 almost exclusively with the case of $K$-equivalent birational spaces.
%

The purpose of this paper is to show that in fact analogous statements 
can be made even when the two spaces are not $K$-equivalent.
In previous work \cite{AS}, the authors conjectured that when $X_+$ and $X_-$ are birational but not $K$-equivalent, 
the respective generating functions for genus zero Gromov--Witten invariants will be related by asymptotic expansion.  This was verified 
in the particular case where the birational transformation was a blowup of $[\CC^N/G]$ at the origin (for $G$ a finite abelian group).  The main result of the current paper verifies the above conjecture in much greater generality, showing that the same type of relation holds when $X_+$ and $X_-$ are toric varieties (or orbifolds) related by variation of GIT as well as for compatible complete intersection varieties $Y_+ \subset X_+$ and $Y_- \subset X_-$.

\subsection{Toric results}
Let $X_+$ and $X_-$ be complete toric orbifolds obtained as GIT quotients of a torus $K$ acting on $\CC^m$, and related by a wall crossing in
the space of stability conditions.  Under the assumption that $X_+$ and $X_-$ are \emph{not} $K$-equivalent, one can construct a common toric blow-up \[\pi_{\pm}: \widetilde X \to X_{\pm}\] such that 
$\pi_+^* (K_{X_+}) - \pi_-^* (K_{X_-})$ is effective.
 (see Section~\ref{s:wc}).

By the work of Coates--Corti--Iritani--Tseng in \cite{CCIT2}, 
there exists an explicit expression for cohomology valued functions $I^+(y, z)$ and $I^-(\tilde y, z)$ which encode the genus zero Gromov--Witten theory of $X_+$ and $X_-$.  
Although not how they are defined a-priori, these $I$-functions can be understood as generating functions for the respective sets of Gromov--Witten invariants.  Our main theorem relates these functions: 
\begin{theorem}[Theorem~\ref{t:finalresult}]  \label{t:main1}
There exists a linear transformation \[L: H^*_{CR}(X_+)[z, z^{-1}] \to H^*_{CR}(X_-)[z, z^{-1}]\] 
such that the power series asymptotic expansion of $L \cdot I^+(y, z)$ in a suitable variable (determined by the wall crossing) recovers the function $I^-(\tilde y, z)$.
%
\end{theorem}

The technical challenge in proving the above theorem is due to the fact that we lack an explicit expression for the asymptotic expansion of $I^+(y,z)$.  
We instead shift our focus from the functions $I^+$ and $I^-$ themselves to the differential equations they satisfy.  In particular, we construct a function $\iI(y,z)$ taking values in $H^*_{CR}(X_-)[z, z^{-1}]$ which has asymptotic expansion equal to $I^-(\tilde y, z)$.  We then show that the components of $\iI(y, z)$ satisfy a system of differential equations, a completion of the GKZ system.  The components of $I^+(y, z)$ give a \emph{basis} of solutions to this system of equations, which then implies the existence of a linear map $L$ 
sending $I^+(y, z)$ to $\iI(y, z)$.

\subsection{Complete intersections}

Given a complete intersection $Y_+ \subset X_+$, the  \emph{quantum Lefschetz principal} (see \cite{Co} for a recent treatment) allows one to construct a function $I^+_Y$ encoding much of the Gromov--Witten theory of $Y_+$.  This function is expressed as an explicit term-by-term modification of the function $I^+$ given above.

In the final section we consider birational spaces $Y_+$ and $Y_-$, defined as complete intersections in  $ X_+$ and $ X_-$ respectively, and compatible with each other in a precise sense.  
We use Theorem~\ref{t:main1} in combination with the  \emph{quantum Lefschetz principal} \cite{Co} to deduce a  relationship between the generating functions $I^+_Y$ and $I^-_Y$.\begin{theorem}[Theorem~\ref{t:ci}]  
There exists a linear transformation \[L: H^*_{CR}(X_+)[z, z^{-1}] \to H^*_{CR}(X_-)[z, z^{-1}]\] 
such that the power series asymptotic expansion of $L \cdot I^+_Y(y, z)$ in a suitable variable (determined by the wall crossing) recovers the function $I^-_Y(\tilde y, z)$.
%
\end{theorem}



\subsection{Relation to the crepant transformation conjecture}

The crepant transformation conjecture has evolved over time in order to incorporate a successively broader range of cases, with the various incarnations usually differing in their viewpoint as to the central object of comparison.
To place Theorem~\ref{t:main1} in its proper context, we briefly recall some of the evolution of the  conjecture.  
%
We will focus here on the genus zero picture.
%

An early form of the conjecture, due to Ruan \cite{Ru}, identifies the quantum cohomology of $K$-equivalent spaces $X_+$ and $X_-$ after analytic continuation in the quantum parameters.  Later a stronger conjecture, relating the full generating function of genus zero invariants, was given by Bryan--Graber \cite{BG} in the case of crepant resolutions of \emph{hard Lefschetz orbifolds}.  A similar statement has also been developed and proven for ordinary flops by Lee--Lin--Wang--et. al. \cite{LLW, ILLW, LLQW}.  
Finally, the most relevant form of the conjecture for our purposes
was proposed by Coates--Iritani--Tseng \cite{CIT} using physical considerations  from mirror symmetry.  This version states that the overruled Lagrangian cones (see \cite{G3} for an introduction to Givental's formalism) of $X_+$ and $X_-$ should be identified after application of a linear symplectic transformation and suitable analytic continuation.  This version of the conjecture is appealing in that not only does it relate the full set of genus zero Gromov--Witten invariants of $X_+$ and $X_-$, but also should apply in the most generality.

In the current work we choose to use the machinery of generating functions (the so-called $I$-functions) rather than Givental's symplectic formalism, although either one fully determines the other (see \cite{CCIT2}).  This has the benefit of yielding a simple Bryan--Graber type correspondence at the level of $I$-functions, although a price is paid.  Namely, extracting information about individual Gromov--Witten invariants from the $I$-function is often quite involved, although possible in principle--see Section~\ref{s:genfunk}.  

A key ingredient in all three versions of the above conjecture is the analytic continuation of certain functions, and thus implicit in the conjecture is the statement that these functions are analytic in some variable.
 When $X_+$ and $X_-$ are $K$-equivalent, $I^+(y, z)$ and $I^-(\tilde y, z)$ can often be shown to be analytic \cite{CIJ} but this fails for a general birational map $X_+ \dasharrow X_-$.  In fact with the setup as in the previous section, $I^+$ will be analytic with an essential singularity at infinity and $I^-$ will only be a formal power series, with radius of convergence zero.  
 
 We propose that in this case analytic continuation is to be replaced by a power series asymptotic expansion of $I^+$, reflecting the inherent asymmetry of the map $X_+ \dasharrow X_-$.  Furthermore the dimension of $H^*_{CR}(X_+)$ will be greater than that of $H^*_{CR}(X_-)$, and so the map $L$ in Theorem~\ref{t:main1} will not be invertible.  The remarkable fact is that these appear to be the only modifications necessary to translate the crepant transformation conjecture into a statement which holds for much more general birational transformations (e.g. blow-ups).\footnote{The existence of a non-invertible transformation $L$ was already conjectured in \cite{Ru} in this general context, although the role of asymptotic expansion had not yet been realized.}

\begin{remark}
We remark finally that a recent paper by Coates--Iritani--Jiang \cite{CIJ} gives yet another perspective on the crepant transformation conjecture, this time in terms of \emph{quantum $D$-modules}.  In \cite{CIJ}, they prove a correspondence between quantum $D$-modules of 
$K$-equivalent toric orbifolds (and complete intersections) under variation of GIT.
Our current paper is indebted to their work, as we use the same variation of GIT setup to conduct our toric wall crossing.
\end{remark}
\begin{remark}
H.~Iritani has also announced results comparing the genus zero Gromov--Witten invariants of $f: Y \to X$ when $f$ is a toric blow-up or flip.  His results are phrased in the language of quantum $D$-modules, and thus differ somewhat from our main theorem.  It will be interesting to compare the two perspectives.  
\end{remark}

\subsection{Acknowledgements}

The authors wish to thank H. Iritani for many useful conversations, for his talks on the crepant transformation conjecture, and for reading an early version of this paper. They thank Y.-P. Lee for expressing interest in the work and for many helpful conversations.  They also wish to thank Y. Ruan for first explaining to them the crepant transformation conjecture, and for teaching them much of what they know about Gromov--Witten theory.   
M. S. was partially supported by NSF RTG Grant DMS-1246989.

\section{Gromov--Witten theory}

In this section we recall those definitions and properties of Gromov--Witten theory which play a role in the correspondence.  This section also serves to set notation.  More details of the theory can be found in \cite{AGV} or \cite{CK}.

\subsection{Invariants}

Let $X$ be a smooth, proper, Deligne--Mumford stack with projective coarse moduli space.  Define the inertia stack
 $IX$ as the fiber product
\[
\xymatrix{
IX \ar[r] \ar[d] & X \ar[d]^\Delta \\
X\ar[r]^\Delta & X \times X \\ 
}
\]
where $\Delta$ is the diagonal map.  $IX$ parametrizes pairs $(x,g)$ where $x \in X$ and $g$ is an element of the isotropy group of $x$.  In the case where $X$ is a stack quotient of the form $[V/G]$ with $V$  a smooth variety and $G$ an abelian group, 
\[ IX = \coprod_{g \in G} [V^g/G],
\]
where $V^g$ denotes the fixed locus of $V$ with respect to $g$.  There is a distinguished component corresponding to $g=  e$ which is isomorphic to $X$.  We call this the \emph{untwisted sector}.  The other components are called
\emph{twisted sectors}.  
In this paper we will be concerned with the case where $G$ is an algebraic torus.  

\begin{notation}
Unless otherwise stated, the coefficients of our cohomology groups are always in $\CC$.
\end{notation}

The relevant cohomology theory for use in orbifold Gromov--Witten theory is Chen--Ruan cohomology:
\begin{definition}[\cite{ChenR1, ChenR2}]
The \emph{Chen--Ruan orbifold cohomology} of $X$ is defined, as a vector space, as:
\[H^*_{CR}(X) := H^*(IX).\]
\end{definition}
The inclusion of the untwisted sector $X \subset IX$ into the inertia stack allows us to view $H^*(X)$ as a summand of $H^*_{CR}(X)$.

There is a natural involution $inv:IX \to IX$ which sends the point $(x,g)$ to $(x, g^{-1})$.  This endows the Chen--Ruan cohomology with a  pairing given by 
\[(\alpha, \beta)_{CR} := \int_{IX} \alpha \cup inv^*(\beta), \hspace{1 cm}  \alpha, \beta \in H^*_{CR}(X).\]

Fix a genus $g$, an integer $n \geq 0$, and a degree $d \in NE(X)_\ZZ = NE(X) \cap H_2(|X|; \ZZ)$; there exists a moduli space $\sMbar_{g, n}(X, d)$ parametrizing stable maps $f: C \to X$, where $C$ is a pre-stable genus-$g$ orbi-curve \cite{AGV} with $n$ marked points, and $f$ is a representable stable map of degree $d$.  The torus action on $X$ induces an action on  $\sMbar_{g, n}(X, d)$.  Although gerbe structure prevents the global existence of evaluation maps $ev_i: \sMbar_{g, n}(X, d) \to IX$, in fact the pullback 
\[ev_i: H^*_{CR}(X) \to H^*(\sMbar_{g, n}(X, d)),\]
can still be defined by identifying the cohomology of $IX$ with that of $\overline{IX}$, the \emph{rigidified inertia stack} (see \cite{AGV} for details).
For each marked point $1 \leq i \leq n$, we define the $i$th $\psi$-class, $\psi_i \in H^*(\sMbar_{g, n}(X, d))$,  to be the first Chern class of the $i$th universal cotangent line bundle $L_i$.  Given classes $\alpha_1, \ldots, \alpha_n \in H^*_{CR}(X)$ and integers $k_1, \ldots, k_n \geq 0$, we define the Gromov--Witten invariant
\[ \br{\alpha_1 \psi_1^{k_1}, \ldots, \alpha_n \psi_n^{k_n}}^X_{g,n,d} := \int_{[\sMbar_{g, n}(X, d)]^{vir}} \bigcup_{i=1}^n ev_i^*(\alpha_i) \cup \psi_i^{k_i}, 
\]
where $[\sMbar_{g, n}(X, d)]^{vir}$ denotes the \emph{virtual fundamental class} on $\sMbar_{g, n}(X, d)$ \cite{AGV}.

\subsection{Quantum cohomology and generating functions}\label{s:genfunk}

Let $\bt = \sum_{i \in I}t^i \phi_i$ represent a general point in  $H^*_{CR}(X)$ for $\{\phi_i\}_{i \in I}$ a choice of basis.  
Given  $\alpha_1, \ldots, \alpha_n \in H^*_{CR}(X)$, and integers $k_1, \ldots, k_n \geq 0$ as above, we define the double bracket
\[\br{\br{ \psi_1^{a_1} \alpha_1, \ldots, \psi_n^{a_n} \alpha_n }}^X:=
\sum_{d \in NE(X)_\ZZ} \sum_{k = 0}^\infty \frac{1}{k!}
\br{ \psi^{a_1} \alpha_1, \ldots, \psi^{a_n} \alpha_n, \bt, \ldots, \bt
}_{0, n+k, d}^X.\]
In general the sum on the right will converge only if we view the parameters $\{t^i\}_{i \in I}$ as formal variables.  We denote $\CC[[\{t^i\}_{i \in I}]]$ simply by $\CC[[\bt]]$.
%

\begin{definition}
The \emph{quantum cohomology} of $X$, 
\[*_{\bt}: H^*_{CR}(X) \times H^*_{CR}(X) \to H^*_{CR}(X)[[ \bt ]],
\]
is defined by the equality
\[ ( \alpha *_{\bt} \beta, \gamma)_{CR} = \br{\br{ \alpha, \beta, \gamma}}^X
\] 
for all $\alpha, \beta, \gamma \in H^*_{CR}(X)$.
\end{definition}
As a consequence of the \emph{WDVV} relations \cite{AGV, CK}, the operation $*_{\bt}$ endows $H^*_{CR}(X)[[ \bt]]$ with  an associative product.  
%

The \emph{ $J$-function}, a cohomology valued generating function of Gromov--Witten invariants of $X$ first introduced by Givental \cite{G1}, is useful in studying quantum cohomology. 
\begin{definition}
The $J$-function of $X$ is defined as 
\[J^X(\bt, z) = z +  \bt + \sum_{i \in I} \br{ \br{ \frac{\phi_i}{z - \psi_1} } }^X \phi^i
\]
where $\{\phi^i\}_{i \in I}$ is the dual basis to $\{\phi_i\}_{i \in I}$ and $1/(z - \psi_1)$ denotes the corresponding expansion in $1/z$.  \end{definition}
The $J$-function naturally takes values in $H^*_{CR}(X)[[ \bt]]((z^{-1}))$.
We note that the
quantum cohomology of $X$ is fully determined by the $J$-function.
The key point is that $J^X$ satisfies the system of partial differential equations
\[ z \frac{\partial}{\partial t^i}\frac{\partial}{\partial t^j} J^X(\bt, z)
= \sum_{k \in I}(\phi_i *_{\bt} \phi_j, \phi^k) \frac{\partial}{\partial t^k}J^X(\bt, z),
\]
which follows from the \emph{topological recursion relations} \cite{CK}.  
The $J$-function can be upgraded to an endomorphism 
${\bf J}^X(\bt, z): H^*_{CR}(X)((z^{-1}))[[\bt]] \to H^*_{CR}(X)((z^{-1}))[[\bt]]$ 
given by:
\[{\bf J}^X(\bt, z): Z(\bt, z) \mapsto 
Z(\bt, z) + \sum_{i \in I} 
\br{ \br{ \frac{\phi_i}{z - \psi_1}, Z(\bt, z)
}}^X \phi^i
.\]
Via the \emph{string equation} \cite{AGV, CK}, $z{\bf J}^X(\bt, z)(1)$ can be seen to equal the original $J$-function.  This motivates the following general definition:
\begin{definition}\label{d:I function}
Let $q^1, \ldots, q^r$ be formal parameters.
An \emph{$I$-function} for $X$ is any cohomology-valued function of the form
\begin{equation}\label{Ifctn}
I^X(\bq, z) = z{\bf J}^X(\btau(\bq), z)(Z(\bq, z)),
\end{equation}
such that $Z(\bq, z) \in H^*_{CR}(X)[z][[\bq]]$ contains only \emph{positive} powers of $z$.
The map $\bq \mapsto \btau(\bq)$ is called the \emph{mirror map}.
\end{definition}

\begin{definition}\label{d:bigI}
We say an $I$-function $I^X(\bq, z)$ is \emph{big} if there exist differential operators $\{P_i(z, z \frac{\partial}{\partial q^j})\}_{i \in I}$ which are polynomial in $z$ and $z \frac{\partial}{\partial q^j}$ such that 
\begin{equation}\label{e:bigbig} z^{-1}P_i\left(I^X(\bq, z)\right) = \phi_i + O(\bq).\end{equation}
\end{definition}     
Using a form of Birkhoff factorization, one can prove (see e.g. \cite{AS, CIJ}):
\begin{lemma}\label{l:BF}
A big $I$-function $I^X(\bq, z)$ explicitly determines the pullback of  $J^X$  under the mirror map.
\end{lemma}

As a consequence, if the image of the mirror map $\bq \to \btau(\bq)$ generates the Chen--Ruan cohomology of $X$, $I^X(\bq, z)$ fully determines the quantum cohomology of $X$.  In many cases an explicit description of the $J$-function is difficult to obtain, and it is more convenient to work with $I$-functions.  This is the approach taken in this paper.

\section{Toric setup}\label{s:3}

We construct our toric varieties as stack quotients via GIT.  The setup is almost identical to that in \cite{CIJ}, so for the readers' convenience we use the same notation.
The initial data consists of 
\begin{itemize}
\item A torus $K \cong (\CC^*)^r$;
\item the lattice $\LL = \Hom(\CC^*, K)$ of co-characters of $K$;
\item a set of characters $D_1, \ldots, D_m \in \LL^\vee = \Hom(K, \CC^*)$;
\item a choice of a \emph{stability condition} $\omega \in \LL^\vee \otimes \RR$.
\end{itemize}
Given the above, the map $(D_1, \ldots, D_m): K \to (\CC^*)^m$ induces an action of  $K$ on $\CC^m$.  
Given $I \subset \set{1, \ldots, m}$, we denote by $\overline I$ the complement of $I$.  Define $\angle I$ to be the subset 
\[\angle I := \set{ \sum_{i \in I} a_i D_i | a_i \in \RR_{>0}} \subset \LL^\vee \otimes \RR,\] and let $(\CC^*)^I \times (\CC)^{\overline I} := \set{ (z_1, \ldots, z_m) | z_i \neq 0 \text{ for } i \in I}.$

 Let $\sA_\omega := \set{I \subset \set{1, \ldots, m}| \omega \in \angle I}$.  Consider the open set \[U_\omega := \bigcup_{I \in \sA_\omega}  (\CC^*)^I \times (\CC)^{\overline I}.\]  We define the toric stack \[X_\omega := [U_\omega / K],\] where brackets denote that we are taking the stack quotient.  


From this perspective, $\LL^\vee \otimes \RR$ may be viewed as the set of stability parameters for the GIT quotient of $\CC^m$ by $K$.  
We restrict ourselves to $\omega$ in the non-negative span $\sum_{i = 1}^m \RR_{\geq 0} D_i$ so that $X_\omega$ is non-empty.  
This set acquires a wall and chamber structure as follows.  For  such $\omega  \in \sum_{i = 1}^m \RR_{\geq 0} D_i$, define
\[C_\omega := \bigcap_{I \in \sA_\omega} \angle I \subset \LL^\vee \otimes \RR.\]   Note that if $\omega ' \in C_\omega$, $X_\omega$ and $X_{\omega '}$ are equal.  Then the \emph{chambers} of $\sum_{i = 1}^m \RR_{\geq 0} D_i$ are $C_\omega$ of maximal dimension $r$, and the \emph{walls} are formed by higher codimension $C_\omega$.  If $C_\omega$ is of maximal dimension, $X_\omega$ will be a Deligne--Mumford stack.  We also assume always that $X_\omega$ is has proper coarse moduli space.

Fix a stability condition $\omega$ such that $C_\omega$ is of maximal dimension.  
The  cohomology  $H^*( X_\omega)$ is generated by the divisor classes $u_i := \{z_i = 0\}$  where $z_i$ is the $i$th coordinate of $\CC^m \supset U_\omega$, viewed as a homogeneous coordinate on $X_\omega$.  The cohomology of $X_\omega$ is computed explicitly in \cite{BCS}, although we will not require it here.
%
%
There exists a linear map $\theta_\omega: \LL^\vee \otimes \RR \to H^2( X_\omega)$ defined such that 
\begin{equation}\label{e:theta}
\theta_\omega(D_i) = u_i ,
\end{equation} this is well defined by Equation~4.7 of \cite{CIJ}.

For $f \in \LL \otimes \QQ$, 
define $I_f := \{ i \in \{1, \ldots, m\}| D_i \cdot f \in \ZZ\} $. Consider the set 
\[ \KK_\omega:=\{f | I_f \in \sA_\omega\} \subset \LL \otimes \QQ.\]  Then $\KK_\omega /\LL$ indexes the fixed point components of $IX_\omega$: $f \in \KK_\omega /\LL$ corresponds to the group element 
\[\exp(f) := (\exp(2 \pi i D_1 \cdot f), \ldots, \exp(2 \pi i D_r \cdot f)) \in K,
\]
which in turn corresponds to the component $X_\omega^f := [U_\omega \cap (\CC^m)^{\exp(f)}/K] \subset IX_\omega$.  The Chen--Ruan orbifold cohomology of $X_\omega$ is then given as a graded vector space by
\[H^*_{CR}(X_\omega) = \bigoplus_{f \in \KK/\LL} H^*(X_\omega^f).\]
Let $\ii_f$ denote the class in $H^*_{CR}(X_\omega)$ corresponding to the fundamental class of $X_\omega^f$.

\subsection{Wall crossing}\label{s:wc}

Choose stability conditions $+$ and $-$ in $\LL^\vee \otimes \RR$ such that $C_+$ and $C_-$ are of maximal dimension and are separated by a codimension one wall.  
Let $W$ denote the hyperplane in $\LL^\vee \otimes \RR$ separating $C_+$ and $C_-$,
and let  $\overline{C_W} := W \cap \overline{C_+} = W \cap \overline{C_-}$ denote the (closure of the) corresponding wall.  Let $e \in \LL$ be a primitive generator of $W^\perp$.  We assume that $X_+$ and $X_-$ are \emph{not} $K$-equivalent, or equivalently, 
that we can choose $e$ such that $ \sum_{j=1}^m D_j \cdot e > 0$.  
We may assume, after possibly switching the labeling of the stability conditions $+$ and $-$, that 
$\omega_+  \cdot e> 0$.  
One can construct (see Section~6.3 of \cite{CIJ}) a common toric blowup
\begin{equation}\label{bigpic}\begin{tikzcd}
& \widetilde X \arrow[swap]{dl}{\pi_+} \arrow{dr}{\pi_-}&\\
X_+ \arrow[dashed]{rr}{\phi}& & x_-,
%
\end{tikzcd}
\end{equation}
such that under our assumptions, $\pi_+^* (K_{X_+}) - \pi_-^* (K_{X_-})$ is effective.  

Let $\widetilde \LL_{\pm}$ denote the lattice in $\LL \otimes \QQ$ generated by $\KK_{\pm}$.  In \cite{CIJ} it is proven that:
\begin{lemma}[\cite{CIJ} Section 5.3]
There exist bases $\{p_1^+, \ldots, p_r^+\}$ and $\{p_1^-, \ldots, p_r^-\}$ of $\widetilde \LL_+^\vee$ and $\widetilde \LL_-^\vee$ respectively such that:
\begin{itemize}
\item $p_i^{\pm}$ lies in $\overline{C_{\pm}}$ for $1 \leq i \leq r$;
\item $p_i^+ = p_i^- \in \overline{C_W}$ for $1 \leq i \leq r-1$;
\end{itemize}
\end{lemma}
For $d $ in $ \LL$, resp. $\widetilde \LL_{\pm}$, let $y^d$ denote the corresponding element in $\CC[\LL]$, resp. $\CC[\widetilde \LL_{\pm}]$.  We have an inclusion \[\CC[C_+^\vee \cap \LL] \to 
\CC[C_+^\vee \cap \widetilde \LL_+] \to \CC[y_1, \ldots, y_r]\] given by $ y^d \mapsto \prod_{i=1}^r y_i^{p_i^+ \cdot d}$. Similarly, the inclusion \[\CC[C_-^\vee \cap \LL] \to  \CC[C_-^\vee \cap \widetilde \LL_-] \to \CC[\tilde y_1, \ldots, \tilde y_r]\] is given by $y^d \mapsto \prod_{i=1}^r \tilde y_i^{p_i^- \cdot d}$.

The coordinates are related by the change of variables
\begin{eqnarray}\label{e:cov}
\tilde y_i &=& y_i y_r^{c_i} \hspace{1 cm} 1 \leq i \leq r-1 \\
\tilde y_r &=& y_r^{-c}  \nonumber
\end{eqnarray}
where  $c = - p_r^+\cdot e /  p_r^- \cdot e  \in \QQ_{>0}$ and $c_i \in \QQ$ are determined by the change of basis from $\{p_i^+\}$ to $\{p_i^-\}$.  Note that $p_r^+\cdot e > 0$ while $p_r^- \cdot e < 0$.

\subsection{The $I$-function}

To simplify notation, for $k \in \LL \otimes \QQ$,
 let $y^k = \prod_{i=1}^r y_i ^{p^+_i \cdot k}$ and  $\tilde y^k = \prod_{i=1}^r \tilde y_i ^{p^-_i \cdot k}$.  Note that $y^k = \tilde y^k$ under the change of variables \eqref{e:cov}, but we distinguish between the two when want to emphasize a particular coordinate system.  
Define \[\sigma_{\pm} := \theta_{\pm}(p_i^{\pm} \log(y_i)) ,\] where $\theta_{\pm}$ is as defined in \eqref{e:theta}.  
Define the functions $I^+(y, z)$ and $I^-(\tilde y, z)$ in $H^*_{CR}(X_+)[z, z^{-1}][\log y_1, \ldots, \log y_r][[y_1, \ldots, y_r]]$ \\
and 
$H^*_{CR}(X_-)[z, z^{-1}][\log \tilde y_1, \ldots, \log \tilde y_r][[\tilde y_1, \ldots, \tilde y_r]]$ respectively by:
\begin{eqnarray}\\
\nonumber
I^+(y ,z) &:=& ze^{\sigma_+/z} \sum_{k \in \KK_+} y^k \left(
\prod_{j=1}^m \frac{\prod_{a: \langle a \rangle = \langle D_j \cdot k \rangle, a \leq 0} (u_j + az)}
{\prod_{a: \langle a \rangle = \langle D_j \cdot k \rangle, a \leq D_j \cdot k} (u_j + az)}
\right) \ii_{[-k]},
\\
I^-(\tilde y ,z) &:=& ze^{\sigma_-/z} \sum_{k \in \KK_-} \tilde y^k \left(
\prod_{j=1}^m \frac{\prod_{a: \langle a \rangle = \langle D_j \cdot k \rangle, a \leq 0} (u_j + az)}
{\prod_{a: \langle a \rangle = \langle D_j \cdot k \rangle, a \leq D_j \cdot k} (u_j + az)}
\right) \ii_{[-k]}.
\nonumber \end{eqnarray}
For $k \in \KK_{\pm}$, let $I^{\pm}_k$ denote the corresponding summand:
\begin{align}I^+( y, z) &= ze^{\sigma_+/z}  \sum_{k \in \KK_+}  y^k I^+_k, \label{e:parts} \\ \nonumber
I^-(\tilde y, z) &= ze^{\sigma_-/z}  \sum_{k \in \KK_-} \tilde y^k I^-_k.
\end{align}

$I^{\pm}$ fully determines the $J$-function $J^{\pm}$ of $X_{\pm}$ in the sense of Definition~\ref{d:bigI}:
\begin{theorem}[\cite{CCIT2}]
$I^{\pm}$ is a big $I$-function for $X_{\pm}$.
\end{theorem}
\begin{proof}
This follows from the mirror theorem for toric stacks as given in \cite{CCIT2} and is explicitly proven in Lemma~5.23 of \cite{CIJ}.
\end{proof}


\subsection{Differential equations}\label{s:diffeq}
In this section we investigate the differential equations satisfied by $I_{\pm}$.  What follows is just a small part of the larger theory of GKZ systems, which can be referenced in \cite{GKZ} or \cite{Ad}.

Since $p_1^{\pm}, \ldots, p_r^{\pm}$ forms a basis for $\LL^\vee \otimes \QQ$, each $D_j$ can be written as a linear combination $D_j = \sum_{i=1}^r a^i_j p_i^+ = \sum_{i=1}^r \tilde a^i_j p_i^-$.  Define the operator
\[\partial_j := \sum_{i=1}^r a^i_j y_i \frac{\partial}{\partial y_i}= \sum_{i=1}^r \tilde  a^i_j \tilde  y_i \frac{\partial}{\partial \tilde y_i}.\]
This operator is designed so that
\[\partial_j y^d =  (D_j \cdot d) y^d.\]
%
Furthermore, we observe that 
\begin{align*} \partial_j e^{\sigma_+/z} &= \partial_j\left(\prod_{i=1}^r y_i^{\theta_+(p_i^+)/z} \right) \\
&= \theta_+(\sum_{i=1}^r a^i_j p_i^+)/z\left(\prod_{i=1}^r y_i^{\theta_+(p_i^+)/z} \right) \\
&= (\theta_+(D_j)/z)e^{\sigma_+/z} \\
&= \frac{u_j}{z} e^{\sigma_+/z}.
\end{align*}
The same is true for the minus side: $\partial_j e^{\sigma_-/z} = \frac{u_j}{z} e^{\sigma_-/z}$.

%
Consider the GKZ system of differential operators $\{\triangle_d\}_{d \in \LL}$
\begin{equation}\label{e:GKZ}\triangle_d := \prod_{j | d_j > 0} \prod_{l=0}^{d_j - 1} \left( z \partial_j   - lz \right) -   y^d \prod_{j | d_j < 0} \prod_{l=0}^{-d_j - 1} \left( z \partial_j   - lz \right) 
\end{equation}
where $d_j := D_j \cdot d$.  One can easily check the following, using the above facts about $\partial_j$.
\begin{proposition}
The relation $\triangle_d I^+(y, z) = 0$ holds for all $d \in \LL$.  
$I^-$ satisfies the same set of equations after replacing $ y^d$ with $\tilde y^d$.
\end{proposition}
%
%
\subsection{Fractional calculus}\label{s:fd}
It will be important in what follows to make use of so-called fractional derivatives \cite{MSH}.  In particular we would like to have a notion of $\left( d / dx \right)^a$ for $a \in \QQ$.  We note that for positive integers $k$ and $a$,
\[ \left(\frac{d}{dx}\right)^a x^k = \frac{k!}{(k-a)!}x^{k-a}.\]  We can generalize this to $a \in \CC$ as
\[ \left(\frac{d}{dx}\right)^a x^k = \frac{\Gamma(k+1)}{\Gamma(k-a+1)} x^{k-a}.\]
To generalize this to other functions of $x$, 
we define, for $0 < a < 1$,
\begin{equation}\label{e:fdiff} \left(\frac{d}{dx}\right)^a f(x) := \frac{1}{\Gamma(1-\alpha)} \frac{d}{dx}\int_0^x \frac{f(t)}{(x-t)^\alpha} dt.\end{equation}
It is not hard to check that this agrees with the previous expression given by Gamma functions when $f(x) = x^k$.
Then for arbitrary $a$, 
\[
\left(\frac{d}{dx}\right)^a f(x) := \left(\frac{d}{dx}\right)^{\langle a\rangle}\circ \left( \left(\frac{d}{dx} \right)^{\lfloor a\rfloor}f(x)\right).\]

\begin{remark}
One must take care with the previous definition.  Unlike the fractional \emph{integral} (see, e.g., \cite{MSH}),  fractional derivatives unfortunately do not satisfy the semigroup property. That is, $\left(\frac{d}{dx}\right)^a\left(\frac{d}{dx}\right)^bf(x)$ does not necessarily equal $\left(\frac{d}{dx}\right)^b\left(\frac{d}{dx}\right)^a f(x)$.  

\end{remark}

\section{Regularization and Watson's lemma}
In this section we construct a function $\iI = \iI(y_1, \ldots, y_r, z)$ which is analytic in $y_r$ and has an asymptotic expansion equal to $I^-$.


\subsection{Regularization}

We begin by noting the following fact:
\begin{proposition}\label{p:an+}
$I^+$ is analytic with respect to $y_r$ with radius of convergence $\infty$.
\end{proposition}

\begin{proof}
Fix $d \in \KK_+$.  We collect together all terms of $I^+(y, z)$ such that the exponent of $y_i$ is $p_i^+ \cdot d$ for $1 \leq i \leq r - 1$ and which are supported on $X_+^{[-d]}$.
The corresponding  function of $y_r$ is (up to a factor of $z e^{\sigma_+/z} y^d$)
\begin{align*}F^+_d(y_r, z) &:=  \sum_{l \in \ZZ} y^{le} I^+_{d+le}.
\end{align*} Note that $y^{le} = y_r^{l(p_r^+ \cdot e)}$ as $p_1^+, \ldots, p_{r-1}^+$ lie in $e^\perp$.  
For $l$ sufficiently small, $I^+_{d + le} = 0$, as $I^+(y, z)$ contains only positive powers of $y_1, \ldots, y_r$ by construction.

The ratio of the $l$th and $l+1$st terms in $F_d(y_r z^{\sum e_j/p_r^+ \cdot e}, z)$ is
\[y_r^{p_r^+ \cdot e} \frac{\prod_{j| e_j < 0} \prod_{m = 0}^{-e_j - 1} \left( u_j/z + k_j + l e_j - m \right) }{\prod_{j| e_j > 0} \prod_{m = 1}^{e_j } \left( u_j/z + k_j + l e_j + m\right)}.
\]
Recall that $\sum e_j > 0$, thus the above expression (expanded in terms of $u_j/z$) goes to zero as $l \mapsto \infty$.  Viewing $z$ as a complex parameter, we conclude by the ratio test that $F_d(y_r z^{\sum e_j/p_r^+ \cdot e}, z)$ and therefore also $F_d(y_r, z)$ are convergent series in $y_r$ with infinite radius of convergence.

\end{proof}

We note that while $I^+$ is analytic in the $y_r$ variable, $I^-$ is not analytic in the $\tilde y_r$ variable.  To compare the two, we first \emph{regularize} the function $I^-$ by adjusting the coefficient of each $\tilde y^k$.  

Define the \emph{regularized} $I^-$ function to be:
\begin{align}\label{e:Ireg} I^-_{reg}(\tilde y ,z) &:= ze^{\sigma_-/z} \sum_{k \in \KK_-}  \frac{\tilde y^k I^-_k}{\Gamma\left(1 - (\sum e_j/p_r^- \cdot e) (p_r^- \cdot k + \theta_-(p_r^-)/z) \right)} 
\end{align}
Note that $I^-_k = 0$ for all $k$ not in $C_-^\vee$.  Because $p_r^- \cdot k \geq 0$ for all $k \in  C_-^\vee$, the above modification is well defined.
%
%
It will be convenient to make the change of variables $x = \tilde y_r^{-(p_r^- \cdot e/\sum e_j)}$ (recall that we've chosen $e$ such that $p_r^- \cdot e < 0$).  View $I^-_{reg} = I^-_{reg}(\tilde y_1, \ldots, \tilde y_{r-1}, x, z)$ as a function of the variables $\tilde y_1, \ldots, \tilde y_{r-1}$ and  $x$.  
In exactly the same manner as the previous proposition, $I^-_{reg}$ can be shown to converge for $x$ within a sufficiently small radius.

For future use we record the following, which holds for all $a > 0$ (see Section~\ref{s:fd}):
\begin{align} \left(\frac{\partial}{\partial x}\right)^a e^{\sigma_-/z} \tilde y^d &= \left(\frac{\partial}{\partial x} \right)^a \left( \left(\prod_{i=1}^{r-1} \tilde y_i^{\theta_-(p_i^-)/z + p_i^- \cdot d}\right) x^{ -(\sum e_j/p_r^- \cdot e)(\theta_-(p_r^-)/z +p_r^- \cdot d)}\right)\nonumber \\ \label{e:dxa}
& = 
\frac{\Gamma\left(1 - (\sum e_j/p_r^- \cdot e) (p_r^- \cdot d + \theta_-(p_r^-)/z) \right)}{\Gamma\left(1 - (\sum e_j/p_r^- \cdot e) (p_r^- \cdot d + \theta_-(p_r^-)/z) - a\right)} \cdot  \\ \nonumber
&e^{\sigma_-/z} \left(\prod_{i=1}^{r-1} \tilde y_i^{p_i^- \cdot d}\right) x^{-(\sum e_j/p_r^- \cdot e)(p_r^- \cdot d) - a}.
\end{align}

\begin{proposition} \label{p:regdif}
For all $d$ in $ \LL$,
$I^-_{reg}$ satisfies the equation $\Delta^{reg}_d I^-_{reg}$ = 0, where $\Delta^{reg}_d$ is defined to be 
\[\prod_{j | d_j > 0} \prod_{l=0}^{d_j - 1} \left( z \partial_j   - lz \right) -  \left(\prod_{i=1}^{r-1} \tilde y_i^{p_i^- \cdot d} \right) \left( \frac{\partial}{\partial x}\right)^{(\sum e_j/p_r^- \cdot e)(p_r^- \cdot d)}  \prod_{j | d_j < 0} \prod_{l=0}^{-d_j - 1} \left( z \partial_j   - lz \right)\]
if $p_r^-\cdot d < 0$, and 
\[\left( \frac{\partial}{\partial x}\right)^{-(\sum e_j/p_r^- \cdot e)(p_r^- \cdot d)} \prod_{j | d_j > 0} \prod_{l=0}^{d_j - 1} \left( z \partial_j   - lz \right) -  \left(\prod_{i=1}^{r-1} \tilde y_i^{p_i^- \cdot d} \right)   \prod_{j | d_j < 0} \prod_{l=0}^{-d_j - 1} \left( z \partial_j   - lz \right)\]
if $p_r^-\cdot d \geq 0$.
\end{proposition}

\begin{proof}
The proof follows by a direct computation using the relations at the beginning of Section \ref{s:diffeq} along with \eqref{e:dxa}.  Although we are viewing $I^-_{reg}$ as an analytic function in $x$, by the same argument as in the previous proposition it converges absolutely for $x$ in a sufficiently small radius.  Thus it suffices to check term-by-term using the power series expansion of \eqref{e:Ireg}, which we may re-write as $I^-_{reg}(\tilde y_1, \ldots \tilde y_{r-1}, x ,z) = \sum_{k \in \KK_-} I^-_{reg, k}$ where 
\[I^-_{reg, k} = ze^{\sigma_-/z} \left(
\prod_{j=1}^m \frac{\prod_{a: \langle a \rangle = \langle k_j \rangle, a \leq 0} (u_j + az)}
{\prod_{a: \langle a \rangle = \langle  k_j \rangle, a \leq k_j} (u_j + az)}
\right)\frac{
\prod_{i=1}^{r-1} \tilde y_i^{ p_i^- \cdot k} x^{ -(\sum e_j/p_r^- \cdot e)(p_r^- \cdot k)} }{\Gamma\left(1 + (\sum e_j/p_r^- \cdot e) (p_r^- \cdot k + \theta_-(p_r^-)) \right)} \ii_{[-k]}.\]

Recall that 
$I^-_{reg, k}$ is $0$ unless $k \in C_-^\vee$.  Fix such a $k \in \KK_-$.  Assume that $p_r^- \cdot d < 0$.  
Applying $\prod_{j | d_j > 0} \prod_{l=0}^{d_j - 1} \left( z \partial_j   - lz \right)$ to $I^-_{reg, k}$ amounts to multiplying by 
\[\prod_{j | d_j \geq 0} \prod_{l=0}^{d_j - 1} \left( u_j  + k_jz - lz \right),\]
which yields
\begin{align}\label{e:LHSdif}
ze^{\sigma_-/z} &\left(
\prod_{j| d_j>0} \frac{\prod_{a: \langle a \rangle = \langle k_j \rangle, a \leq 0} (u_j + az)}
{\prod_{a: \langle a \rangle = \langle  k_j \rangle, a \leq k_j-d_j} (u_j + az)}
\right)
 \left(
\prod_{j| d_j<0} \frac{\prod_{a: \langle a \rangle = \langle k_j \rangle, a \leq 0} (u_j + az)}
{\prod_{a: \langle a \rangle = \langle  k_j \rangle, a \leq k_j} (u_j + az)}
\right) \cdot
\\ \nonumber
&\frac{\prod_{i=1}^{r-1} \tilde y_i^{ p_i^- \cdot k} x^{ -(\sum e_j/p_r^- \cdot e)(p_r^- \cdot k)} }{\Gamma\left(1 + (\sum e_j/p_r^- \cdot e) (p_r^- \cdot k + \theta_-(p_r^-)) \right)} \ii_{[-k]}.
\end{align}
Note that the above expression is $0$ if $k-d \notin C_-^\vee$, since it is supported on the subspace of $X_+^{[-k]}$ given by $z_j = 0$ for all $j$ such that $k_j - d_j \in \ZZ_{<0}$. 
On the other hand, if $k-d \in C_-^\vee$, by applying \[ \left(\prod_{i=1}^{r-1} \tilde y_i^{p_i^- \cdot d} \right) \left( \frac{\partial}{\partial x}\right)^{(\sum e_j/p_r^- \cdot e)(p_r^- \cdot d)}  \prod_{j | d_j < 0} \prod_{l=0}^{-d_j - 1} \left( z \partial_j   - lz \right)\] to $I^-_{reg, k-d}$ we obtain
\begin{align} \label{e:RHSdif}
&\left(
\prod_{j| d_j>0} \frac{\prod_{a: \langle a \rangle = \langle k_j \rangle, a \leq 0} (u_j + az)}
{\prod_{a: \langle a \rangle = \langle  k_j \rangle, a \leq k_j-d_j} (u_j + az)}
\right)
 \left(
\prod_{j| d_j<0} \frac{\prod_{a: \langle a \rangle = \langle k_j \rangle, a \leq 0} (u_j + az)}
{\prod_{a: \langle a \rangle = \langle  k_j \rangle, a \leq k_j} (u_j + az)}
\right) \cdot
\\ \nonumber
&\left(\prod_{i=1}^{r-1} \tilde y_i^{p_i^- \cdot d} \right) \left( \frac{\partial}{\partial x}\right)^{(\sum e_j/p_r^- \cdot e)(p_r^- \cdot d)}
ze^{\sigma_-/z} 
\frac{\prod_{i=1}^{r-1} \tilde y_i^{ p_i^- \cdot (k-d)} x^{ -(\sum e_j/p_r^- \cdot e)(p_r^- \cdot (k -d))} }{\Gamma\left(1 + (\sum e_j/p_r^- \cdot e)(p_r^- \cdot (k-d) + \theta_-(p_r^-)) \right)} \ii_{[-k]}.
\end{align}
By \eqref{e:dxa} this is seen to equal \eqref{e:LHSdif}.  The calculation for $p_r^- \cdot d > 0$ is similar.

\end{proof}

\begin{proposition}\label{p:Ireg} The regularized $I$-function $I^-_{reg}(\tilde y_1, \ldots \tilde y_{r-1}, x ,z)$ is analytic with respect to $x$, and can be analytically continued to all but a finite number of points in the region $\text{arg}(x) < \pi$.  Furthermore, after fixing powers of $y_1, \ldots, y_{r-1}$, the corresponding coefficient of $I^-_{reg}(\tilde y_1, \ldots \tilde y_{r-1}, x ,z)$ is bounded by an algebraic function as $x \mapsto \infty$.
\end{proposition}
\begin{proof}

Similarly to  Proposition \ref{p:an+}, we fix $d \in \KK_-$, and group together the terms of $I^-_{reg}(\tilde y_1, \ldots \tilde y_{r-1}, x ,z)$ corresponding to degree $d - le$ for $l \in \ZZ$:
\begin{align*}F^-_{reg, d}(x, z) &:= \sum_{l \in \ZZ}  
I^-_{reg, d-le}.
\end{align*} Again $I^-_{reg, d-le} = 0$ for sufficiently small $l$.  


Consider the differential operator $\Delta_e^{reg}$ from the above proposition:

\begin{align*} &\prod_{j | e_j > 0} \prod_{l=0}^{e_j - 1} \left(  \frac{-(p_r^- \cdot e) e_j}{\sum_m e_m}  x \frac{\partial}{\partial x} +d_j + u_j  - l \right) -  \\  &\left( \frac{\partial}{\partial x}\right)^{\sum e_j} \prod_{j | e_j < 0} \prod_{l=0}^{-e_j - 1} \left( \frac{-(p_r^- \cdot e) e_j}{\sum_m e_m}  x \frac{\partial}{\partial x} + d_j  + u_j - l \right)\end{align*}
we see that this annihilates $F^-_{reg, d}(x, z) $.
The corresponding differential equation has singularities at $0, \infty$, and all $x$ satisfying $\left( \tfrac{-p_r^- \cdot e}{\sum e_j}x\right)^{\sum e_j} = \prod_{\{j | e_j \neq 0\}} e_j^{-e_j}$.  The singularity at $x = \infty$ is regular, which implies that for each $d \in \KK_-$, $F^-_{reg, d}(x, z)$ is bounded by an algebraic function as $x \mapsto \infty$.  %
%
%
%

\end{proof}

\subsection{Laplace transform and Watson's lemma}

Define a new function $\iI(\tilde y_1, \ldots, \tilde y_{r-1}, u,z )$ using the Laplace transform:

\begin{align}\label{e:laptrans} \iI(\tilde y_1, \ldots, \tilde y_{r-1}, u,z ) : = & u\cL \left(I^-_{reg}(\tilde y_1, \ldots, \tilde y_{r-1}, x, z)\right) \\ \nonumber =& u\int_0^\infty e^{- u x} I^-_{reg}(\tilde y_1, \ldots, \tilde y_{r-1}, x, z) dx.\end{align}
Here again we are viewing $I^-_{reg}$ as an analytic function in $x$.  We choose our path of integration to be any ray from $0$ to $\infty$ with \[0<\text{arg}(x) < \text{min}(\pi/2, \pi/\sum e_j)\] to avoid the singular points found in Proposition~\ref{p:Ireg}.
By Proposition~\ref{p:Ireg}, after fixing powers of $\tilde y_1, \ldots, \tilde y_{r-1}$, the corresponding coefficient of $I^-_{reg}(\tilde y_1, \ldots \tilde y_{r-1}, x ,z)$ is in $O(x^N)$ for some $N>0$ as $x \mapsto \infty$.  
%
Thus the integral is well defined for all $u$ such that $\mathfrak{Re}(ux)>0$.  In fact the path integral does not depend on the choice of ray as long as $\text{arg}(x)$ is within the range given above.  The function $\iI(\tilde y_1, \ldots, \tilde y_{r-1}, u,z )$ is defined for all $u$ such that $|\text{arg}(u)| < \pi/2$.


\begin{proposition}\label{p:4.4} The asymptotic expansion of $\iI$ is given by $I_-$:
\[\iI(\tilde y_1, \ldots, \tilde y_{r-1}, u,z ) \sim I^-(\tilde y_1, \ldots, \tilde y_{r-1}, \tilde y_r, z)|_{\tilde y_r =   u^{(\sum e_j/p_r^- \cdot e)}}\] as $u \mapsto \infty$ along any ray satisfying $|\arg(u)|<\pi/2$.

\end{proposition}

\begin{proof}
First, view the classes $u_i$  ($1 \leq i \leq m$) as formal parameters.  In this case
the corresponding asymptotic expansion is a direct application of Watson's Lemma.
One version of Watson's Lemma states that given a function $\phi(x) = x^\lambda g(x)$ where: 
\begin{itemize}
\item $g(x)$ is infinitely differentiable at $x = 0$ and $g(0) \neq 0$; 
\item $\mathfrak{Re}(\lambda) > -1$;
\item there exists a $b>0$ such that $|\phi(x)| < e^{bx}$ for $|x|$ sufficiently large;
\end{itemize}
then the following asymptotic expansion holds:
\[u\int_0^\infty e^{-ux} \phi(x) dx \sim \sum_{n=0}^\infty \frac{g^{(n)}(0) \Gamma(1+\lambda + n)}{n! u^{\lambda + n}}\]
as $u \mapsto \infty$.
One proof of this lemma is given in \cite{Miller}, Proposition 2.1.  Although the conditions on $\phi(x)$ are slightly different in \cite{Miller} than those stated above 
in fact the same proof implies the result in this case as well (see \cite[Section 2.3]{Miller} or \cite{Tem} for the version given above).
%

To apply the lemma, group the terms of $I^-$ according to their powers of $x$ modulo $\sum e_j$.  Write $I^-$ as a sum of $N = {-p_r^- \cdot e}$ functions, 
$L_0, \ldots, L_{-p_r^- \cdot e - 1}$, where the function $L_l$ contains those terms of $I^-$ with $x$-powers of the form $x^{(\sum e_j/p_r^- \cdot e)(\theta_-(p_r^-)/z + l) + k\sum e_j}$ for  $k$ in $\ZZ$.  In other words, $L_l(x,z)$ can be written in the form $x^{(\sum e_j/p_r^- \cdot e)(\theta_-(p_r^-)/z + l)}g(x)$ where $g(x)$ is a power series in $x$ and $g(0) \neq 0$.  Let $L_{l, reg}$ denote the corresponding summand of $I^-_{reg}$ and let $\mathbb{L}_l := u \cL ( L_{l, reg}(x,z))$. Watson's Lemma implies that $\mathbb{L}_l \sim L_l$ for each $0 \leq l < -p_r^- \cdot e - 1$.  Summing over $l$ then gives the result when the cohomology classes $u_i$ are viewed as formal parameters.

To then prove the statement in $H^*_{CR}(X_-)$, we use the above, and note that for each $k \in \KK_-/\LL$, $H^*\big(X_-^{[-k]}\big)$ is a quotient of the algebra $\CC[ u_1, \ldots, u_n]$ (see e.g. \cite{BCS}).
\end{proof}


\begin{proposition}

$\iI(\tilde y_1, \ldots, \tilde y_{r-1}, u,z )$ satisfies the differential equation
\begin{equation}\label{e:lapdelta} \prod_{j | d_j > 0} \prod_{l=0}^{d_j - 1} \left( z \bar \partial_j   - lz \right)\iI -  \left(\prod_{i=1}^{r-1} \tilde y_i^{p_i^- \cdot d} \right) u^{(\sum e_j/p_r^- \cdot e)(p_r^- \cdot d)}  \prod_{j | d_j < 0} \prod_{l=0}^{-d_j - 1} \left( z \bar \partial_j   - lz \right) \iI = 0,
\end{equation}
where  
\[\bar \partial_j :=  \sum_{i=1}^{r-1} \tilde a^i_j \tilde  y_i \frac{\partial}{\partial \tilde y_i} + \tilde a_r^j\frac{p_r^- \cdot e} {\sum e_j} u\frac{\partial }{\partial u}.\]
\end{proposition}

\begin{proof}
The proof follows by applying $u \cc L$ to the equation $\Delta^{reg}_d I^-_{reg} = 0$ of Proposition~\ref{p:regdif} and using various well-known properties of the Laplace transform.  

We recall the following  facts about the Laplace transform.  First, if $f(x)$ and $g(x)$ are  functions with well defined Laplace transforms and $f(0) = 0$, then:
\begin{enumerate}
\item $\cc L(x  f(x)) = -\frac{d}{du} \cc L(f(x))$;
\item $\cc L (\frac{d}{dx} f(x)) = u \cc L( f(x))$;
\item $\cc L( (f*g)(x) ) = \cc L(f(x)) \cc L(g(x))$;
\end{enumerate}
where $(f*g)(x)$ denotes the convolution of $f$ and $g$.  Let $F(u) = u\cc L( f(x))$, then properties (1) and (2) together imply that 
\begin{equation}\label{e:lapdif1} u \cc L\left(x \frac{d}{dx} f(x)\right) = - u \frac{d}{du} F(u).\end{equation}
Furthermore, by (1), (2), and (3) above, together with definition~\eqref{e:fdiff}, one can check that for any $a  \in \QQ$,
\begin{equation}\label{e:lapdif2} u\cc L \left( \left(\frac{d}{dx}\right)^a f(x) \right) = u^a F(x),\end{equation}
provided that the left hand side is well defined and that $\left(\frac{d}{dx}\right)^{a-n} f(x)|_{x = 0} = 0$ for all integers $1 \leq n \leq \lfloor a + 1\rfloor$.

It is not \emph{a priori} obvious that \eqref{e:lapdif1} holds when applied to $I^-_{reg}$, and in fact the left hand side of \eqref{e:lapdif2} will not always be defined for $f(x)$ equal to $I^-_{reg}$.  Nonetheless we will show that equations~\eqref{e:lapdif1} and~\eqref{e:lapdif2} hold true in the cases of relevance to us.
%
%

We first check 
\eqref{e:lapdif1}. One can use the power series expansion~\eqref{e:Ireg}, together with the fact that 
\begin{equation}\label{e:lapform} u\cc L \left( x^{\lambda+a} \right) = \Gamma(1 + a + \lambda) u^{-\lambda-a},\end{equation}
where both the right and left hand side are meant to represent power series expansions in $\lambda$, for $\lambda$ a formal parameter and $a \geq 0$.  This implies that 
\begin{equation}\label{e:dlhs} u \cc L \prod_{j | d_j > 0} \prod_{l=0}^{d_j - 1} \left( z \partial_j   - lz \right) I^-_{reg} = 
\prod_{j | d_j > 0} \prod_{l=0}^{d_j - 1} \left( z \bar \partial_j   - lz \right) u \cc L \left(I^-_{reg}\right).\end{equation}

Checking \eqref{e:lapdif2} is slightly more subtle, since in general the Laplace transform of $\left(\frac{d}{dx}\right)^a I_{-, reg}(x) $ is not defined.  However consider, for instance, the operator 
 \[\left(\prod_{i=1}^{r-1} \tilde y_i^{p_i^- \cdot d} \right) \left( \frac{\partial}{\partial x}\right)^{(\sum e_j/p_r^- \cdot e)(p_r^- \cdot d)}  \prod_{j | d_j < 0} \prod_{l=0}^{-d_j - 1} \left( z \partial_j   - lz \right),\]
coming from the right hand side of $\Delta^{reg}_d$,  for some $d \in \LL$ such that 
$p_r^-\cdot d < 0$.  From equation~\eqref{e:RHSdif}, one sees that $\prod_{j | d_j < 0} \prod_{l=0}^{-d_j - 1} \left( z \partial_j   - lz \right) I^-_{reg, k -d}$ vanishes unless $k \in C_-^\vee$.  This implies that $\prod_{j | d_j < 0} \prod_{l=0}^{-d_j - 1} \left( z \partial_j   - lz \right) I^-_{reg, k} = 0$ for all $k$ such that $p_r^- \cdot d > p_r^- \cdot k$.  We deduce that the expression
 \[
 \prod_{j | d_j < 0} \prod_{l=0}^{-d_j - 1} \left( z \partial_j   - lz \right) I^-_{reg}(x, z)\]
 can be expanded as $e^{\sigma_-/z} x^{(\sum e_j/p_r^- \cdot e)(p_r^- \cdot d)} f(x,z)$, where $f(x,z)$
   is a power series in $x^{-1/p_r^- \cdot e}$ with positive radius of convergence.  Applying the operator  $\left( \frac{\partial}{\partial x}\right)^{(\sum e_j/p_r^- \cdot e)(p_r^- \cdot d)}$, the result takes the form $e^{\sigma_-/z} g(x,z)$ for $g(x,z)$ a power series in $x^{-1/p_r^- \cdot e}$.  This has a well defined Laplace transform.
Using equation \eqref{e:lapdif2}, one can then conclude that 
\begin{align} \label{e:drhs} &u \cc L\left( \left( \frac{\partial}{\partial x}\right)^{(\sum e_j/p_r^- \cdot e)(p_r^- \cdot d)}  \prod_{j | d_j < 0} \prod_{l=0}^{-d_j - 1} \left( z \partial_j   - lz \right) I^-_{reg}\right) \\ =&u^{(\sum e_j/p_r^- \cdot e)(p_r^- \cdot d)} \cdot u \cc L \left( \prod_{j | d_j < 0} \prod_{l=0}^{-d_j - 1} \left( z \partial_j   - lz \right) I^-_{reg} \right) \nonumber \\ \nonumber
= & u^{(\sum e_j/p_r^- \cdot e)(p_r^- \cdot d)}  \prod_{j | d_j < 0} \prod_{l=0}^{-d_j - 1} \left( z \bar \partial_j   - lz \right) u \cc L \left(I^-_{reg}\right).
\end{align}

Equations~\eqref{e:dlhs} and~\eqref{e:drhs} together imply~\eqref{e:lapdelta} in the case where $p_r^- \cdot d < 0$.  The other case is similar.

%
\end{proof}

Consider now the change of variables 
\begin{eqnarray}\label{e:cov2}
\tilde y_i &=& y_i y_r^{c_i} \hspace{1.5 cm} 1 \leq i \leq r-1 \\
 u &=& y_r^{-c p_r^- \cdot e/ \sum e_j}  \nonumber.
\end{eqnarray}
View the function $\iI( y_1, \ldots,  y_{r-1}, y_r,z )$ as a function of the variables $y_1, \ldots, y_r$.  Note that under this change of variables, 
\[\left(\prod_{i=1}^{r-1} \tilde y_i^{p_i^- \cdot d} \right) \left(u\right)^{(\sum e_j/p_r^- \cdot e)(p_r^- \cdot d)} \mapsto \left(\prod_{i=1}^{r-1}  y_i^{p_i^- \cdot d} y_r^{c_i p_i^- \cdot d} \right) y_r^{-c(p_r^- \cdot d)} = \prod_{i=1}^r y_i^{p_i^+ \cdot d} = y^d.\]  
Similarly, 
\[\bar \partial_j = \sum_{i=1}^{r-1} \tilde a^i_j \tilde  y_i \frac{\partial}{\partial \tilde y_i} +\tilde a_r^j\frac{p_r^- \cdot e} {\sum e_j} u\frac{\partial }{\partial u} \mapsto \sum_{i=1}^{r}  a^i_j   y_i \frac{\partial}{\partial  y_i} = \partial_j
\] 
Under this change of variables, the previous proposition states the following.
\begin{proposition}\label{p:finally}
For all $d \in \LL$, $\iI( y_1, \ldots, y_r,z )$ satisfies the equation  $\Delta_d \iI(y, z) = 0$, where $\Delta_d$ is defined as in \eqref{e:GKZ}.
\end{proposition}

\section{The correspondence}
From the previous section we know that $\iI(y, z)$ has asymptotic expansion equal to $I^-(\tilde y, z)$ as $y_r \mapsto \infty$ along a suitable ray.  
In this section we deduce the main theorem by proving the existence of a linear transformation $L:  H^*_{CR}(X_+)[z, z^{-1}] \to H^*_{CR}(X_-)[z, z^{-1}]$ which maps $I^+(y,z)$ to $\iI(y, z)$.  

\subsection{The weak Fano case}
For simplicity, we first consider the case where $X_+$ is \emph{extended weak Fano}.  
\begin{definition}
The Deligne--Mumford stack $X$ is said to be \emph{weak Fano} if $\rho = c_1(TX)$ is in the closure of the K\"ahler cone.  

Given a toric stack $X_\omega = [U_\omega / K]$ defined via the characters $D_1, \ldots, D_m
\in \LL^\vee$ and the stability condition $\omega \in \LL^\vee \otimes \RR$, we say $X_\omega$ is \emph{extended weak Fano} if 
\[
\hat \rho:= \sum_{i = 1}^m D_i \in \overline{C_\omega}.
\]
\end{definition}
The notion of extended weak Fano depends upon the choice of presentation of $X_\omega$ as a toric quotient, however it is only slightly stronger than weak Fano.

The reason to consider this case is a due to a simplification arising from the following result of Iritani.
\begin{proposition}
When $\hat \rho \in \overline{C_+}$, the components of $I^+(y, z)$ give a basis of solutions to the GKZ system of equations $\{\Delta_d f(y, z) = 0\}_{d \in \LL}$, where $\Delta_d$ is defined in \eqref{e:GKZ}.
\end{proposition}
\begin{proof}
This follows immediately from Lemma 3.9 and Proposition 4.4 of \cite{Iri}.
\end{proof}
We immediately deduce the main result in the case where $X_+$ is extended weak Fano.
\begin{theorem}\label{t:weakfano}
There exists a linear transformation $L: H^*_{CR}(X_+)[z, z^{-1}] \to H^*_{CR}(X_-)[z, z^{-1}]$
such that $L \cdot I^+(y, z) = \iI(y, z) \sim I^-(\tilde y, z)$ as $y_r \mapsto \infty$ along any ray with $|\arg(y_r)|$ sufficiently close to 0. 
\end{theorem}
\begin{proof}
By Proposition~\ref{p:finally}, the components of $\iI(y, z)$ give solutions to the GKZ system of equations.  By the previous proposition,
each component of $\iI(y, z)$ must therefore be a linear combination of the components of $I^+(y, z)$.  This uniquely defines the map $L: H^*_{CR}(X_+)[z, z^{-1}] \to H^*_{CR}(X_-)[z, z^{-1}]$.

The second claim, that $\iI(y, z) \sim I^-(\tilde y, z)$ as $y_r \mapsto \infty$, is Proposition~\ref{p:4.4}.  See \eqref{e:cov2} for the change of variables between $u$ and $y_r$.
\end{proof}

\subsection{The general case}
When $X_+$ is not extended weak Fano, the dimension of the GKZ system is larger than $\text{dim}(H^*_{CR}(X_+))$, and so the $I$-function $I^+$ no longer generates all solutions.
%
In this more general setting, one must consider the completion of the corresponding $D$-module at $y_1= \cdots = y_r = 0$ to recover the full set of differential equations satisfied by $I^+$.  The proof closely follows the argument of Iritani given in Section 4.2 of~\cite{Iri}, but with the crucial adjustment to the non-weak Fano case provided by a dimension bound from Gonzalez--Woodward \cite{GW2}.
\begin{notation} To simplify the expressions to come, we will use the following notational substitution when defining various rings and modules: $ y := y_1, \ldots, y_r$;
$y^{-1} := y_1^{-1}, \ldots, y_r^{-1}$; $ \tfrac{\partial}{\partial y}:=\tfrac{\partial}{\partial y_1}, \ldots, \tfrac{\partial}{\partial y_r}$ etc...  Let $D$ denote the ring of differential operators on $\CC[y]$ with the usual relations.
%
%
\end{notation}

Let $\mathscr{I}_{GKZ}$ denote the ideal in $ \CC[z][y, y^{-1}]\langle zy \tfrac{\partial}{\partial y} \rangle$
generated by the differential operators $\Delta_d$ of \eqref{e:GKZ}.
Define the $D$-module $M_{GKZ} := \CC[z][y, y^{-1}]\langle zy \partial/\partial y \rangle/ \mathscr{I}_{GKZ}$.

Let $\mathscr{I}_{poly}$ denote the intersection $\mathscr{I}_{GKZ} \cap \CC[z][y]\langle zy \tfrac{\partial}{\partial y} \rangle$ and consider the corresponding $D$-module  $M_{poly} := \CC[z][y]\langle z y \tfrac{\partial}{\partial y} \rangle/ \mathscr{I}_{poly}$.  Let $\widehat M_{poly}$ denote $\CC[z][[y]]\langle zy \tfrac{\partial}{\partial y} \rangle/ \overline{\mathscr{I}_{poly}}$, where $\overline{\mathscr{I}_{poly}}$ is the closure of $\mathscr{I}_{poly}$ in $\CC[z][[y]]\langle zy \tfrac{\partial}{\partial y} \rangle$ in the $y$-adic topology.

By Proposition 4.4 of \cite{Iri}, there exists a Zariski open subset $\cc M^\circ$ of 
$\spec \CC[y, y^{-1}]$
 on which $M_{GKZ}^\circ := \s O_{\cc M^\circ}[z] \otimes_{\CC[z, y, y^{-1}]} M_{GKZ}  $ is finitely generated as an $\s O_{\cc M^\circ}[z]$-module.  By Lemma 3.8 of \cite{Iri}, we may assume, after possibly shrinking $\cc M^\circ$, that $\spec \CC[y] \setminus \cc M^\circ$ is a union of hypersurfaces $V' \cup V''$, where every irreducible component of $V'$ avoids the origin, and $V''$ is given by the union of the coordinate hyperplanes $\{y_1 \cdots y_r = 0\}$.  Then $\cc M^\circ =  \cc M' \cap \spec \CC[y, y^{-1}]$, where $\cc M' =\cc M \setminus V'$.  
 Note that $\widehat{ \Gamma(\s O_{\cc M '})}$, the ring of formal functions on $\cc M '$ in a neighborhood of the origin, is equal to $\widehat{\CC[y]} = \CC[[y]]$.

\begin{proposition}  The $D$-module $\widehat M_{poly}^\circ:= \widehat M_{poly}[y^{-1}]$
is a coherent sheaf of rank at most $\dim \left( H^*_{CR}(X_+)\right)$ over $\spec( \CC[z][[y]][y^{-1}])$.
\end{proposition}

\begin{proof}

$M_{GKZ}^\circ$ is finitely generated as an $\s O_{\cc M^\circ}$-module.  
Note that $M_{GKZ}^\circ$ is equal to $M_{poly}^\circ := \s O_{\cc M^\circ}[z] \otimes_{\CC[z, y]} M_{poly}  $, since $M_{GKZ}$ equals $M_{poly}$ over $\spec \CC[y, y^{-1}] \supset \cc M^\circ$.
Thus $M_{poly}^\circ$ is finitely generated as an $\s O_{\cc M^\circ}[z]$-module, and therefore $ \CC[z][[y]]\otimes_{\s O_{\cc M '}[z]} M_{poly}^\circ $ is finitely generated as a $ \CC[z][[y]]\otimes_{\s O_{\cc M ' }[z]} \s O_{\cc M^\circ}[z]$-module.  By the surjection $\CC[z][[y]] \otimes_{\s O_{\cc M '}[z]} M_{poly}^\circ   \to \widehat M_{poly}^\circ$ we conclude that $\widehat M_{poly}^\circ$ is finitely generated as a  $ \CC[z][[y]]\otimes_{\s O_{\cc M ' }[z]} \s O_{\cc M^\circ}[z] = \CC[z][[y]][y^{-1}]$-module.
%
%

Thus $\widehat{M}_{poly}^\circ$ is a coherent sheaf over $\spec(\CC[z][[y]][y^{-1}]  )$.  $\widehat{M}_{poly}^\circ$ is also endowed with a flat connection away from $z = 0$ as described in section 4.2 of \cite{Iri}.  Therefore $\widehat{M}_{poly}^\circ$ is locally free away from $z = 0$.  The restriction to $z = 0$ of $\widehat{M}_{poly}^\circ$ is isomorphic to the completion of the Batyrev ring: $\CC[[y]][y^{-1}] \otimes_{\CC[y]} B(X_+)  $ (see \cite{Ba} or \cite{GW2} for the definition of the Batyrev ring $B(X_+)$). 
By Theorem 4.23 of Gonzalez-Woodward \cite{GW2}, this has rank equal to  $\dim \left(H^*_{CR}(X_+) \right)$.
By Nakayama's lemma, the fiber over every point of $\spec( \CC[z][[y]][y^{-1}] )$
has dimension at most $\dim \left( H^*_{CR}(X_+)\right)$
\end{proof}

Due to the presence of powers of $\log(y_j)/z$ in the $I$-function $I^+(y, z)$, we must extend scalars in order to compare $I^+(y, z)$ to the GKZ $D$-module.  
Define \[\widehat{M}^\circ_{poly}[\log(y)] : = \CC[y, z, 1/z, \log(y)] \otimes_{\CC[y, z]} \widehat{M}^\circ_{poly}.\]
Then $\widehat{M}^\circ_{poly}[\log(y)]$ is a coherent sheaf of rank at most $\text{dim}(H^*_{CR}(X_+))$ over $\spec(  \CC[z, 1/z] [[y]] [y^{-1}, \log(y)])$.   $\widehat{M}^\circ_{poly}[\log(y)]$ also acquires the structure of a 
$D$-module by defining \[ \frac{\partial }{\partial y_j} \log(y_m) := \frac{\delta_{j,m}}{y_m}.\]

Define a map of $\CC[z, 1/z] [[y]] [y^{-1}, \log(y)]\otimes_{\CC[y]} D$-modules 
\[L^+: \widehat{M}^\circ_{poly}[\log(y)] \to H^*_{CR}(X_+) \otimes \CC[z, 1/z] [[y]] [y^{-1}, \log(y)]\] 
 by sending $P(z, y, \log(y), zy \tfrac{\partial}{\partial y}) \mapsto P(z, y, \log(y), zy \tfrac{\partial}{\partial y})\left(I^+(y, z)\right)$.  Since $I^+$ satisfies $P \cdot (I^+) = 0$ for all $P \in \s I_{poly}$, it follows by continuity that $I^+$ is annihilated by all of $ \overline{ \s I_{poly}}$.  Since $I^+$ is a \emph{big} $I$-function (see \eqref{e:bigbig}), the map $L^+$ is surjective.  Comparing ranks by the above proposition we conclude the following:
%
%
%
%
%
\begin{corollary}\label{c:2tolast}
The map 
\[L^+: \widehat{M}^\circ_{poly}[\log(y)] \to H^*_{CR}(X_+) \otimes \CC[z, 1/z] [[y]] [y^{-1}, \log(y)]\] defined above is an isomorphism of $\CC[z, 1/z] [[y]] [y^{-1}, \log(y)] \otimes_{\CC[y]} D$-modules.
%
\end{corollary}
We arrive at our main result.
\begin{theorem}\label{t:finalresult}
There exists a linear transformation $L: H^*_{CR}(X_+)[z, z^{-1}] \to H^*_{CR}(X_-)[z, z^{-1}]$
such that $L \cdot I^+(y, z) = \iI(y, z) \sim I^-(\tilde y, z)$ as $y_r \mapsto \infty$ along a ray with $|\arg(y_r)|$ sufficiently close to 0.  
\end{theorem}

\begin{proof}
By the same argument as in the previous corollary, there is a well-defined map of $\CC[z, 1/z] [[y]] [y^{-1}, \log(y)] \otimes_{\CC[y]} D$-modules 
\[L^-: \widehat{M}^\circ_{poly}[\log(y)] \to H^*_{CR}(X_-) \otimes \CC[z, 1/z] [[y]] [y^{-1}, \log(y)]\]
 defined by $P(z, y, \log(y), zy \tfrac{\partial}{\partial y}) \mapsto P(z, y, \log(y), zy \tfrac{\partial}{\partial y})\left( \iI(y, z) \right)$.  Define $L$ to be the composition $L^- \circ (L^+)^{-1}$, viewed as a map of modules (rather than sheaves).  Since 
both $L^+$ and $L^-$ are maps of $D$-modules, 
 $ \tfrac{\partial }{ \partial y_j}  L^{\pm} = L^{\pm}\tfrac{\partial }{ \partial y_j} $ for $1 \leq j \leq r$.  
 One can check that a function $f(y, \log(y), z) \in \CC[z, 1/z] [[y]] [y^{-1}, \log(y)]$ is constant if and only if $\tfrac{\partial}{\partial y_j} f(y, \log(y), z) = 0$ for $1 \leq j \leq r$. 
We conclude that $L^+$ and $L^-$ are both $\CC[z, z^{-1}]$-linear maps, and therefore so is $L$.

The second claim, that $\iI(y, z) \sim I^-(\tilde y, z)$ as $y_r \mapsto \infty$, is Proposition~\ref{p:4.4}.

%
\end{proof}

\section{Complete intersections}
In this section we use the above results combined with the quantum Lefschetz principle to deduce corresponding statements for certain birational complete intersections  $Y_+ \subset X_+$ and $Y_- \subset X_-$.

Given $\phi: X_+  \dasharrow X_-$ as in Section~\ref{s:wc} and a choice of characters $E_1, \ldots, E_n \in \LL^\vee$ (see Section~\ref{s:3} for notation), one obtains corresponding line bundles $L_{1, \pm}, \ldots, L_{n, \pm}$ over $X_\pm$.  Let $E_+  = \oplus_j L_{j, +} \to X_+$ and $E_- = \oplus L_{j, -} \to X_-$ denote vector bundles obtained as the direct sum of the $+/-$ line bundles.  Let $s_+$ and $s_-$ be regular sections of $E_+$ and $E_-$ respectively, and let $Y_+ \subset X_+$ and $Y_- \subset X_-$. 
denote the corresponding substacks.

Following \cite{CIJ}, we require that:
\begin{enumerate}
\item For $1 \leq j \leq n$, $E_j$ lies in $\overline C_W$;
\item For $1 \leq j \leq n$, $L_{j, \pm}$ is pulled back from the coarse space $|X_{\pm}|$;
\item $s_+$ and $s_-$ are compatible with the map $\phi: X_+  \dasharrow X_-$;
\end{enumerate}
In the above situation, $\phi$ induces a birational map $\phi_Y: Y_+ \dasharrow Y_-$.

For $1 \leq j \leq n$, let  $v_j \in H^*(X_\pm)$ denote $c_1(L_{j, \pm})$ (by abuse of notation we will let $v_j$ denote the cohomology class in either $X_+$ or $X_-$).
We define the \emph{$E_\pm$-twisted} $I$-function on $X_\pm$ as 
\begin{align}\label{e:Itwisted} 
I^+_{tw}( y ,z) &:= ze^{\sigma_+/z} \sum_{k \in \KK_+}   y^k I^+_k
\left(  \prod_{j=1}^n \prod_{a=1}^{E_j \cdot k} (v_j + az)  \right) \\ \nonumber
I^-_{tw}( \tilde y ,z) &:= ze^{\sigma_-/z} \sum_{k \in \KK_-}  \tilde y^k I^-_k
\left(  \prod_{j=1}^n \prod_{a=1}^{E_j \cdot k} (v_j + az)  \right). 
 \end{align}
 The twisted $I$-functions are almost $I$-functions for $Y_+$ and $Y_-$, except that they take values in the Chen--Ruan cohomology of $X_\pm$ rather than $Y_\pm$.  Their relation to the Gromov--Witten theory of $Y_\pm$ comes via the so-called \emph{quantum Lefschetz principle}, which can be phrased as follows.  Let 
\begin{align}\label{e:IY} 
I^+_{Y}( y ,z) &:= e(E_+) I^+_{tw}(y, z) = \left(\prod_{j=1}^n v_j \right) I^+_{tw}(y, z) \\ \nonumber
I^-_{Y}( \tilde y ,z) &:= e(E_-) I^-_{tw}( \tilde y ,z) = \left(\prod_{j=1}^n v_j \right) I^-_{tw}( \tilde y ,z). 
 \end{align}
 \begin{theorem}(\cite{CCIT3}) 
$ I^\pm_{Y}$ is the pushforward of an $I$-function on $Y_\pm$ via the map $i_\pm: Y_\pm \to X_\pm$.  More precisely, there exists a function $Z(y, z) \in H^*_{CR}(X_+)[z][[y]]$ and a map $y \mapsto \btau (y) \in H^*_{CR}(X_+)$ such that 
\[I^+_{Y}( y ,z) = (i_+)_* \left(  z{\bf J}^{Y_+}(i_+^*(\btau(y)), z)(i_+^*Z(y, z)) \right).\] The analogous statement holds for $I^-_Y$.
 \end{theorem}
 
 \begin{proof}
 By Corollary 23 of \cite{CCIT3}, there exist $Z(y, z) \in H^*_{CR}(X_+)[z][[y]]$ and $y \mapsto \btau (y) \in H^*_{CR}(X_+)$ such that $I^+_{tw}( y ,z)$ can be written as $z{\bf J}^{+}_{tw}(\btau(y), z)(Z(y, z))$, where $J^+_{tw}$ is the $E_+$-twisted $J$-function on $X_+$ (see \cite{CCIT3} for a discussion).  By Theorem 1.1 in \cite{Co}, 
 \[i_+^*\left( z{\bf J}^{+}_{tw}(\btau(y), z)(Z(y, z))\right)  = z{\bf J}^{Y_+}(i_+^*(\btau(y)), z)(i_+^*Z(y, z)).\]  The conclusion then follows since $(i_+)_* \circ i_+^* (\alpha) = e(E_+) \alpha$ for $\alpha \in H^*_{CR}(X_+)$.
 \end{proof}
 
 As with the definition of $\partial_j$ in Section~\ref{s:diffeq}, for each $1 \leq j \leq n$, we can define a differential operator $\bar \partial_j$ corresponding to $E_j$.  Each $E_j$ can be written as a linear combination $E_j = \sum_{i=1}^r b^i_j p_i^+ = \sum_{i=1}^r \tilde b^i_j p_i^-$.  Define the operator
\[\bar \partial_j := \sum_{i=1}^r b^i_j y_i \frac{\partial}{\partial y_i}= \sum_{i=1}^r \tilde  b^i_j \tilde  y_i \frac{\partial}{\partial \tilde y_i}.\]  Then as in Section~\ref{s:diffeq}, 
\begin{align}
\bar \partial_j y^d &= (E_j \cdot d) y^d; \label{e:dy1} \\ \label{e:dy2}
\bar \partial_j e^{\sigma_\pm/z} &= \frac{v_j}{z }e^{\sigma_\pm/z}.
\end{align}
 We note that $E_j \in \overline C_W$ for $1 \leq j \leq n$, thus $b^r_j = \tilde b^r_j = 0$.  
 
 As one might expect, we can express the terms of the $E_\pm$-twisted $I$-function in terms of certain differential operators applied to the terms of $I^\pm$.  It is convenient to group the summands of the $I$-function in terms of powers of $y_1, \ldots, y_{r-1}$.  For $[k] \in (\LL \otimes \QQ) / \langle e \rangle$, define 
 \[ G_{[k]}^+(y, z) := ze^{\sigma_+/z}\sum_{\stackrel{l \in \KK_+}{l \equiv k \mod e}} y^l I_l^+,
 \]
 where $I_l^+$ is the $l$th coefficient, as in \eqref{e:parts}.  Then
 \[ I^+(y,z) = \sum_{[k] \in \KK_+/ \langle e \rangle} G_{[k]}^+(y, z).
 \]
 We define $G^-_{[k]}(\tilde y, z)$ similarly.  Note that $l \equiv k \mod e$ if and only if $p_i^\pm \cdot l  = p_i^\pm \cdot k$ for $1 \leq i \leq r-1$ (recall that $p_i^+ = p_i^-$ for $1 \leq i \leq r-1$).  Thus $G_{[k]}^+(y, z)$ contains exactly those summands of $I^+(y,z)$ such that the exponent of $y_i$ is $p_i^+ \cdot k$ for $1 \leq i \leq r-1$.  
 
 The linear transformation $L: H^*_{CR}(X_+)[z, z^{-1}] \to H^*_{CR}(X_-)[z, z^{-1}]$ 
 preserves the power of $y_i$ for $1 \leq i \leq r-1$.  The change of variables from $y$ to $\tilde y$ maps $\tilde y_i$ to $ y_i \cdot y_r^{c_i}$ for some $c_i$. and the asymptotic expansion treats $\tilde y_1, \ldots, \tilde y_{r-1}$ as constant.   Define $\mathbb{G}_{[k]}(y,z) := L (G^-_{[k]}(\tilde y, z))$.  By keeping track of  powers of $y_1, \ldots, y_{r-1}$, these observations allow us to refine the main result for toric varieties, Theorem~\ref{t:finalresult}, to the following:
 \begin{lemma}\label{l:mainrefined}
 For each $k \in \LL\otimes \QQ$, the linear transformation $L$ of Theorem~\ref{t:finalresult} 
 maps $G_{[k]}^+(y, z)$ to the function $\mathbb{G}_{[k]}(y,z)$, which has asymptotic expansion $G_{[k]}^-(\tilde y, z)$ as $y_r \mapsto \infty$.
 \end{lemma}
 
Next we make the following observation, which follows immediately from \eqref{e:dy1} and \eqref{e:dy2} applied to the expression for $I_Y^\pm$.
\begin{align}\label{e:IYbetta} 
I^+_{Y}( y ,z) &:= 
\sum_{[k] \in \KK_+/ \langle e \rangle} \prod_{j=1}^n \prod_{a = 0}^{E_j \cdot k} \left( z \bar \partial_j + az \right) G_{[k]}^+(y, z)
 \\ \nonumber
I^-_{Y}( \tilde y ,z) &:= 
\sum_{[k] \in \KK_-/ \langle e \rangle} \prod_{j=1}^n \prod_{a = 0}^{E_j \cdot k} \left( z \bar \partial_j + az \right) G_{[k]}^-(\tilde y, z).
 \end{align}

Applying the operator $\prod_{j=1}^n \prod_{a = 0}^{E_j \cdot k} \left( z \bar \partial_j + az \right)$ to the relationship in Lemma~\ref{l:mainrefined}, we conclude that the same relationship between the $I$-functions for $X_+$ and $X_-$ holds for the (pushforwards of the) $I$-functions for $Y_+$ and $Y_-$.
 \begin{theorem}\label{t:ci}
 The linear transformation $L: H^*_{CR}(X_+)[z, z^{-1}] \to H^*_{CR}(X_-)[z, z^{-1}]$ of Theorem~\ref{t:finalresult} maps $I^+_{Y}( y ,z)$ to a function $\iI_Y(y, z)$ with asymptotic expansion $I^-_Y(\tilde y, z)$ as $y_r \mapsto \infty$.
 \end{theorem}

\begin{proof}
The operator $\prod_{a = 0}^{E_j \cdot k} \left( z \bar \partial_j + az \right)$ commutes with the linear transformation $L$, as well as with the change of variables in \eqref{e:cov2} and the asymptotic expansion $y_r \mapsto \infty$.  We conclude that 
\begin{align*}
L \left(\prod_{j=1}^n \prod_{a = 0}^{E_j \cdot k} \left( z \bar \partial_j + az \right) G_{[k]}^+(y, z)\right) &=\prod_{j=1}^n \prod_{a = 0}^{E_j \cdot k} \left( z \bar \partial_j + az \right)  \mathbb{G}_{[k]}(y,z) \\
& \sim \prod_{j=1}^n \prod_{a = 0}^{E_j \cdot k} \left( z \bar \partial_j + az \right) G_{[k]}^-(\tilde y, z).
\end{align*}
Summing over all $k \in \LL \otimes \QQ$, it follows that $L \left( I^+_{Y}( y ,z) \right) \sim I^+_{Y}( y ,z)$.  
\end{proof}

%

\bibliographystyle{plain}
\bibliography{references}

\end{document}